\documentclass[11pt]{amsart}
\usepackage{amsmath}
\usepackage{amsfonts}
\usepackage[mathscr]{eucal}
\usepackage{amssymb}
\newtheorem{thm}{Theorem}[section]

\newtheorem{cor}[thm]{Corollary}

\newtheorem{prop}[thm]{Proposition}

\theoremstyle{definition}

\numberwithin{equation}{section}

\newcommand{\bb}[1]{\mathbb{#1}}

\begin{document}

\title[Hardy and Hilbert]{Extensions of the Inequalities of Hardy and HIlbert}

\author[V.~I.~Paulsen]{Vern I.~Paulsen}
\address{Department of Mathematics, University of Houston,
Houston, Texas 77204-3476, U.S.A.}
\email{vern@math.uh.edu}
\author[D.~Singh]{Dinesh Singh}
\address{Department of Mathematics, University of Delhi, Delhi, India}

\thanks{Research supported in part by NSF }

\begin{abstract}
In this note we produce generalized versions of the classical inequalities of Hardy and of Hilbert and we establish their equivalence. Our methods rely on the $H^1-BMOA$ duality. We produce a class of examples to establish that the generalizations are non-trivial.
\end{abstract}

\maketitle

\section{Introduction}
The classical inequality of Hardy states that there exists a constant
$A$ so that if $f(z) =
\sum_{n=0}^{\infty} a_n z^n$ is a function in the Hardy space
$H^1(\bb D),$ then
\begin{equation}\label{hardy} \sum_{n=0}^{\infty} \frac{|a_n|}{n+1} \le A
  \|f\|_1.\end{equation}
The well known inequality of Hilbert states that there exists a
constant $B$ so that if $(\alpha_n), (\beta_n)$ are two sequences in
the HIlbert space $\ell^2_+= \ell^2(\bb Z^+),$ then
\begin{equation}\label{hilbert} \sum_{m=0}^{\infty} \sum_{n=0}^{\infty}
  \frac{|\alpha_n||\beta_n|}{n+m+1} \le B
  \|(\alpha_n)\|_2\|(\beta_n)\|_2. \end{equation}

 In this paper we establish
generalizations of both of these inequalities by replacing  the
sequences $\left\{\frac{1}{n}\right\}$ and $\left\{\frac{1}{n+m+1}\right\}$ with a larger class
of sequences. We then prove that the two generalized
inequalities are equivalent.  In particular, it follows that the two
classical inequalities (\ref{hardy}) and (\ref{hilbert}) are
equivalent.  The literature frequently mentions the fact that Hilbert's
inequalitiy (\ref{hilbert}) implies Hardy's inequality (\ref{hardy}),
but we found no mention of the converse.

The inequalities of Hardy~(\ref{hardy}) and Hilbert~(\ref{hilbert})
are extremely important for a variety of reasons and they are still of current interest. For some classic expositions, we refer the
reader to  \cite{Du}, \cite{Ga}, \cite{HLP} and \cite{He}.
For a sampling of the current research interest, we refer the reader to \cite{MV}, \cite{Ol}, \cite{Sa}, \cite{SY}, \cite{Ya}, and \cite{zhu}.

\section{Preliminary Results}

The Lebesque spaces on the unit circle $\bb T$ in the complex plane
shall be denoted by $L^p,$ $ 1 \le p \le \infty.$  The Hardy spaces
$H^p$ are the subspaces consisting of all $f$ in $L^p$ whose negative
Fourier coefficients vanish. In particular, the Hilbert
space $H^2$ will often be identified with $\ell^2_+,$ via the map $e_n
\to z^n, n \ge 0.$ The $H^p$ spaces can be identified as spaces of
holomorphic functions on the open unit disk $\bb D$ that satisfy the
condition:
\begin{equation} \sup_{0<r<1} \left( \frac{1}{2 \pi} \int_0^{2\pi} |f(re^{i \theta})|^p d
\theta\right)^{1/p} <  \infty, \end{equation}
and for $f \in H^p$ this supremum is equal to the norm of $f$ in
$L^p.$

The dual of $H^1$ is the Banach space $BMOA$ defined as all functions
$f \in H^1$ such that
\[ \|f\|_* = \sup_I I(|f - I(f)|) < \infty,\] where $I$ varies over
all subarcs of $\bb T,$ $|I|$ denotes the normalized arclength of $I$ and
\[ I(f) = \frac{1}{|I|} \int_I f(\theta) d \theta.\]
The norm of a function $f \in BMOA$ is given by  
\begin{equation} \|f\| = |f(0)| + \|f\|_*.\end{equation}
We recall that by the famous duality theorem of Fefferman, $BMOA$ can
be identified isometrically and conjugate linearly with the dual of $H^1$ via the dual pairing given by
\[ \langle f,p \rangle = \frac{1}{2\pi} \int_0^{2\pi} p(e^{i \theta})
\overline{f(e^{i \theta})} d \theta,\]
for $f \in BMOA$ and $p \in H^1$ a polynomial.

Finally, we shall also need Fefferman's second characterization of
BMOA. Namely, that a function $f \in H^1$ is in $BMOA$ if and only if
$(1-r^2)|f^{\prime}(r e^{i \theta})|^2 dr d \theta$ is a Carleson
measure on $\bb D.$

For the details of these facts, we refer the reader to \cite{Du},
\cite{Ga}, and \cite{HLP}.

\section{Statement of the Main Results}

We state our three main results below, together with their immediate
corollaries, but we defer their proofs until later sections.
We let $X$ denote the space of sequences $\{c_n\}$ such that
\begin{equation} \|\{c_n\}\|^2 = \sup \left\{ \frac{\sum_{k=0}^n (k+1)^2 |c_k|^2}{n+1}: n \in \bb N \right\}<\infty.\end{equation}
It is not hard to see that $X$ is a Banach space in this
norm and that $c_k = \frac{1}{k+1}$ is an element of unit norm in this space.

\begin{thm}[Generalized Hardy Inequality] There is a constant $A$ so
  that if $(c_n)$ is a sequence in
  $X,$
then for any $f(z) = \sum_{k=0}^{\infty} a_n z^n$ in $H^1,$
\[ \sum_{n=0}^{\infty} |a_nc_n| \le A \|(c_n)\| \|f\|_1. \]
\end{thm}


\begin{cor}[Hardy's Inequality] There is a constant $\delta$ such that
  for any $f(z) = \sum_{n=0}^{\infty} a_n z^n$ in $H^1,$
\[ \sum_{n=0}^{\infty} \frac{|a_n|}{n+1} \le \delta \|f\|_1.\]
\end{cor}
\begin{proof} Set $c_n = \frac{1}{n+1}$ and apply Theorem~3.1.
\end{proof}

\begin{thm}[Generalized Hilbert Inequality] There is a constant $B$ so
  that if $(c_n) \in X$ and $(a_n), (b_n)$ in $\ell^2_+,$ then
\[ \sum_{n,m=0}^{\infty} |a_n| |b_m| | c_{n+m}| \le B \|(c_n)\| \|(a_n)\|_2\|(b_n)\|.\]
\end{thm}


\begin{cor}[Hilbert's Inequality] There is a constant $B$ so that for
  any $(a_n), (b_n)$ in $\ell^2_+,$
\[\sum_{m,n=0}^{\infty} \frac{|a_n||b_m|}{n+m+1} \le B
\|(a_n)\|_2\|(b_n)\|_2.\]
\end{cor}
\begin{proof} Apply Theorem~3.3 with $c_n = \frac{1}{n+1}.$
\end{proof}

\begin{thm} For each $(c_n)$ in $X,$ the corresponding
  generalized Hardy inequality is equivalent to the corresponding
  generalized Hilbert inequality. Moreover, the least constant $A$ in
  the generalized Hardy inequality is equal to the least constant $B$
  in the generalized Hilbert inequality.
\end{thm}
\begin{cor} The Hardy inequality and Hilbert inequality are
  equivalent.
\end{cor}

\section{Proof of Theorem 3.5}

We present the proof of Theorem 3.5 first because it is the most direct and demonstrates the equivalence of the classic Hardy and Hilbert inequalities.

We first prove that the generalized Hilbert inequality implies the generalized Hardy inequality. Let $(a_n)$ and $(b_n)$ be in $\ell^2_+$ so that by the generalized Hilbert inequality(Theorem 3.3),
\[ \sum_{n,m=0}^{\infty} |a_n| |b_n| | c_{n+m}| \le B \|(c_n)\| \|(a_n)\|_2 \|(b_n)\|_2,\]
where $\|(c_n)\|$ is defined by (3.1).

Let $f(z) = \sum_{n=0}^{\infty} \alpha_n z^n$ be in $H^1.$ It is well known that $f$ can be factorized as $f(z) = g(z)h(z)$ where $g(z) = \sum_{n=0}^{\infty} \beta_n z^n$ and $h(z) = \sum_{n=0}^{\infty} \gamma_n z^n$ are both in $H^2$ and $\|f\|_1= \|g\|_2 \|h\|_2.$  Applying the generlized Hilbert inequality to the sequences $(\beta_n)$ and $(\gamma_n)$ we have that
\[ \sum_{n,m=0}^{\infty} |\beta_n| |\gamma_n| |c_{n+m}| \le B \|(c_n)\| \|(\beta_n)\|_2 \|(\gamma_n)\|_2.\]
This inequality may be rewritten as
\[ \sum_{m=0}^{\infty} \big( \sum_{k=0}^m |c_m| |\beta_k| | \gamma_{m-k}| \big) \le B \|(c_n)\| \|g\|_2\|h\|_2 = B\|(c_n)\| \|f_1\|.\]
By the Cauchy product formula,
\[ |\alpha_m| = | \sum_{k=0}^m \beta_k \gamma_{m-k}| \le \sum_{k=0}^m |\beta_k||\gamma_{m-k}|,\]
and hence,
\[ \sum_{m=0}^{\infty} |\alpha_m| |c_m| \le B \|(c_m)\| \|f\|_1,\]
so that the generalized Hardy inequality follows. This also shows that
$A \le B.$

We now prove that the generalized Hardy inequality implies the
generalized Hilbert inequality.  Note that to prove the generalized
Hilbert inequality it is sufficient to consider sequences in $\ell^2_+$
of non-negative terms. Choosing any two sequences, $(a_n), (b_n)$ in $\ell^2_+$ with $a_n \ge 0$ and $b_n \ge 0.$  Setting $g(z) = \sum_{n=0}^{\infty} a_n z^n$ and $h(z) = \sum_{n=0}^{\infty} b_n z^n$ we obtain two functions in $H^2.$ Hence the function $f(z) = g(z) h(z)= \sum_{n=0}^{\infty} d_n z^n$ is in $H^1,$ where $d_n = \sum_{k=0}^n a_k b_{n-k}.$

Thus, by the generalized Hardy inequality(Theorem 3.1),
\[ \sum_{n=0}^{\infty} |d_n| |c_n| \le A \|(c_n)\| \|f\|_1 \le A
\|(c_n)\| \|g\|_2\|h\|_2= A \|(c_n)\| \|(a_n)\|_2 \|(b_n)\|_2.\]
Substituting for $d_n$ we have
\[\sum_{n,m=0}^{\infty} |c_{n+m}| a_n b_m =
\sum_{n=0}^{\infty}\big(\sum_{k=0}^n |c_n| a_{k} b_{n-k} \big) \le A
\|(c_n)\| \|(a_n)\|_2 \|(b_n)\|_2,\]
and so the generalized Hilbert inequality follows. Moreover, we see
that $B \le A$ and so the best constants must be equal.

\section{Proofs of Theorems 3.1 and 3.3}

We begin with the proof of the generalized Hardy inequality. 
Note that
$$ K = \sup_{0\le r <1} \left(\frac{r}{1-r^{2[\frac{1}{1-r}]}}\right)^2 $$
is finite,
where $[t]$ denotes the greatest integer less than or equal to $t.$
 
Given $I$ an arc on the unit circle, we let
\[ R(I) = \{ z \in \bb D:  \frac{z}{|z|} \in I, 1 - |z| \le |I| \},\]
where $|I|= \frac{length(I)}{2 \pi}$ denotes the normalized arc length.
Recall that a positive Borel measure $\mu$ on $\bb D$ is called a {\bf Carleson measure} if there is a constant $k$ so that $\mu(R(I)) \le k |I|$ for all subarcs.

The theorems of
Fefferman and Garsia say that a function $f$ in $H^1$ belongs to $BMOA$ if and only if
\[ (1-r^2) |f^{\prime}(r e^{i \theta})|^2 r dr d \theta \]
is a Carleson measure. Moreover, in this case if we let $\eta(f)^2$ be
the least constant $k$ in the Carleson condition, then $\eta(f)$ is
essentially, the Garsia norm and hence is a
norm on $BMOA$ that is equivalent to $\|f\|_*.$
 
\begin{prop} Let $(c_n) \in X,$ and set $g(z) = \sum_{n=0}^{\infty} c_n z^n,$ then $g \in BMOA$ and
\[\int \int_{R(I)} (1-r^2) |g^{\prime}(r e^{i \theta})|^2 r dr d
\theta \le 2K \|(c_n)\|^2 |I|.\]
Consequently, the linear map $L:X \to BMOA$ defined by $L((c_n))= g(z)$ is bounded.
\end{prop}
\begin{proof}Set $A= \|(c_n)\|^2$ so that $\sum_{k=m}^{m+n} k^2 |c_k|^2 \le An.$
Note that for $0 < r <1,$
\begin{multline*} \int_0^{2\pi} |g^{\prime}(re^{i \theta})|^2 d \theta = \sum_{k=1}^{\infty} k^2|c_k|^2r^{2k} = \sum_{p=0}^{\infty}\sum_{m= p[\frac{1}{1-r}] +1}^{(p+1)[\frac{1}{1-r}]} m^2|c_m|^2 r^{2m} \le \\ \sum_{p=0}^{\infty}r^{2p[\frac{1}{1-r}]+2}A(p+1)[\frac{1}{1-r}] \le A[\frac{1}{1-r}]r^2(\frac{1}{1-r^{2[\frac{1}{1-r}]}})^2 \le AK \frac{1}{1-r}. \end{multline*}

Hence, 
\[\int\int_{R(I)} (1-r^2) |g^{\prime}(re^{i \theta})|^2 r dr d \theta \le AK \int_{1-|I|}^1r+r^2 dr \le 2AK|I|\]
and we have shown the Carleson estimate. Hence $g \in BMOA$ with
$\eta(g) \le \sqrt{2K} \|(c_n)\|.$
\end{proof}

{\em Proof of the generalized Hardy inequality.} Let $(c_n) \in X,$ and let $f(z) = \sum_{n=0}^{\infty} a_n z^n \in H^1.$ Write $a_n = \alpha_n|a_n|$ where $|\alpha_n| =1.$  Set $g(z) = \sum_{n=0}^{\infty} \alpha_n|c_n|z^n$ so that by Proposition~4.1, $g$ is in $BMOA$ with norm independent of the particular choice of $\alpha$'s.

We have that
\[ \langle f,g \rangle = \lim_{r \to 1^-} \int_0^{2\pi} f(re^{i
  \theta}) \overline{g(re^{i \theta})} d \theta =
\sum_{n=0}^{\infty} |a_n||c_n|. \]

Hence, 
\[\sum_{n=0}^{\infty} |a_n| |c_n| \le \|f\|_1\|g\|_* \le \|L\|
\|f\|_1 \|(c_n)\|\]
and the proof is complete.

{\em Proof of the generalized Hilbert Inequality.} Given $(a_n), (b_n)
\in \ell^2_+$ and $(c_n) \in X,$ set $f(z) = \sum_{n=0}^{\infty}
|a_n|z^n$ and $h(z) = \sum_{n=0}^{\infty} |b_n| z^n.$  Then $f,h$ are in
$H^2$ and hence, $f(z)h(z) = \sum_{n,m=0}^{\infty} |a_nb_m| z^{n+m}$ is
in $H^1.$ 

Defining $g(z) = \sum_{n=0}^{\infty} |c_n| z^n$ and arguing as above,
we have that
\begin{multline*} \sum_{n,m=0}^{\infty} |a_n||b_m||c_{n+m}| = \langle
  fh,g\rangle \le\\ \|L\|\|(c_n)\|\|fh\|_1 \le \|L\|
  \|(c_n)\|\|(a_n)\|_2\|(b_n)\|_2, \end{multline*}
and the result follows.

\section{Slow Decay in X}

The importance of the classic Hardy inequality, which corresponds to
the sequence $c_n = \frac{1}{n+1},$ is that it gives us
information about the rate of decay of coefficients of $H^1$
functions. For this reason it is interesting to seek elements $(c_n)
\in X$
which tend to 0 at a slower rate. In particular, we shall show that
there exist sequences in $X$ so that $\sup \{ n^s c_n: n \in \bb N \}=
 \infty$ for any $s>1/2.$

To this end fix an $r, 1/2 \le r \le 1$ and fix $\beta > 1.$ We
define $c_n$ inductively as follows:
Set $c_1 =1$ so that $1^2 |c_1|^2 \le 1 \beta$ and $1^2 |c_1|^2 +
\frac{(1+1)^2}{(1+1)^{2r}}\le (1+1) \beta$ Assume that $c_1,...,c_n$
have been defined so that 
\begin{equation}
\sum_{k=1}^m k^2 |c_k|^2 \le m \beta \text{ for } m=1,2,...,n.
\end{equation}

If $\sum_{k=1}^n k^2|c_k|^2 + \frac{(n+1)^2}{(n+1)^{2r}} \le \beta(n+1),$ then we set
$c_{n+1} = \frac{1}{(n+1)^r}$ and we have that (6.1) holds for $m=1,..., (n+1).$
If this inequality is false, then we set $c_{n+1} = \frac{1}{n+1}$ and observe that
$\sum_{k=1}^n k^2 |c_k|^2 + (n+1)^2 |c_{n+1}|^2 \le \beta n + 1 \le \beta(n+1),$ so again (6.1) holds for $m= 1,..., (n+1).$

In this way we define a sequence $c_n$ that satisfies (6.1) for all $n$ with $c_n = \frac{1}{n^r}$ or $c_n = \frac{1}{n}$ for each $n.$

Now suppose that $c_k= \frac{1}{k}$ for $n \le k \le n+m.$  Then
\begin{multline*} \sum_{k=1}^{n+m} k^2 |c_k|^2 + \frac{(n+m+1)^2}{(n+m+1)^{2r}}   = \sum_{k=1}^n k^2 |c_k|^2 + (m-1) + (n+m+1)^{2-2r} \\ \le \beta n + (m-1) + (n+m+1)^{2-2r}.\end{multline*}
Since, 
\[ \beta n + (m-1) + (n+m-1)^{2-2r} \le \beta(n+m+1),\]
if and only if $(n+m-1)^{2-2r} \le (\beta -1) m + \beta +1$ and $2-2r < 1$ we see that eventually, there will exist an $m$ for which this inequality holds and then $c_{n+m+1} = \frac{1}{(n+m+1)^r}.$

Thus,  $c_n= \frac{1}{n^r}$ infinitely often.  In particular, $\sup \{ n^s c_n: n \in \bb N \}$ will be infinite for any $s>r\ge 1/2.$

\end{document}